\documentclass[12pt]{amsart}
\newcommand\C{{\mathbb C}}

\newtheorem {theo}{Theorem}
\newtheorem {coro}{Corollary}
\newtheorem {lemm}{Lemma}

\newtheorem {prop}{Proposition}

\def\oz1{d{\overline z}^1}
\def\oz2{d{\overline z}^2}
\def\oz3{d{\overline z}^3}

\def\oI{\overline I}

\def\oz{\overline z}

\def\oIq1{\oI_1\cdots\oI_{q-1}}
\def\oIq2{\oI_1\cdots\oI_{q-2}}

\def\vol{{\mbox{vol}}}

\usepackage{etoolbox}
\patchcmd{\section}{\scshape}{\bfseries}{}{}
\makeatletter
\renewcommand{\@secnumfont}{\bfseries}
\makeatother

\newcommand{\conj}[1]{\overline{#1}}
\newcommand{\xy}{x,\conj{y}}
\newcommand{\xx}{x,\conj{x}}
\newcommand{\yy}{y,\conj{y}}
\newcommand{\alphabeta}{\alpha\conj{\beta}}
\newcommand{\yx}{y,\conj{x}}
\newcommand{\Modulus}[1]{\left|#1\right|}

\newcommand{\integration}[4]{\int_{#2} #4 \, #3}
\newcommand{\norm}[1]{\| #1 \|_2}

\newcommand{\diff}{{\mathrm{d}}}
\newcommand{\canCoord}[1]{z}
\newcommand{\schFn}[1]{\Psi_{#1}}
\newcommand{\sect}{s}

\newcommand{\rad}{r}

\newcommand{\ptNorm}[2]{\|{#1(#2)}\|}

\newenvironment{nouppercase}{%
  \renewcommand{\uppercasenonmath}[1]{}}{}

\begin{document}
\title[~]
{Eventual positivity of Hermitian algebraic functions \\ and associated integral operators}
\author[~]
{Colin Tan and Wing-Keung To}

\address{Colin Tan, Department of Mathematics, National University of Singapore, Block
S17, 10 Lower Kent Ridge Road, Singapore 119076}
\email{colinwytan@gmail.com}
\address{Wing-Keung To,
Department of Mathematics, National University of Singapore, Block
S17, 10 Lower Kent Ridge Road, Singapore 119076}
\email{mattowk@nus.edu.sg}

\thanks{Wing-Keung To was partially supported by the research grant
R-146-000-142-112 from the National University of Singapore and
the Ministry of Education.
}

\keywords{
Hermitian algebraic functions, integral operators, positivity
}
\subjclass[2010]{
32L05, 32A26, 32H02
}
\begin{nouppercase}
\maketitle
\end{nouppercase}
\numberwithin{equation}{section}
\begin{abstract}
 Quillen proved that repeated multiplication of the standard sesquilinear form to a positive Hermitian bihomogeneous polynomial eventually results in a sum of Hermitian squares, which was the first Hermitian analogue of Hilbert's seventeenth problem in the nondegenerate case.  Later Catlin-D'Angelo generalized this positivstellensatz of Quillen to the case of Hermitian algebraic functions on holomorphic line bundles over compact complex manifolds by proving the eventual positivity of an associated integral operator.  The arguments of Catlin-D'Angelo, as well as that of a subsequent refinement by Varolin, involve subtle asymptotic estimates of the Bergman kernel.  In this article, we give an elementary and geometric proof of the eventual positivity of this integral operator, thereby yielding another proof of the corresponding positivstellensatz.
 \end{abstract}

\section{{\bf Introduction}}\label{introduction and Statement of Results}

\medskip
A central topic in real algebraic geometry is Hilbert's seventeenth problem of representing a nonnegative form on $\mathbb R^n$ as a sum of squares of rational functions.
An affirmative solution was first provided by Artin's positivstellsatz in 1927 \cite{Art27}.
Since then, related topics have continued to be widely studied from different viewpoints (cf. \cite{Qui68},
\cite{Rez95}, \cite{CLPR96}, \cite{CD97}, \cite{CD99}, \cite{TY06},
\cite{Var08}, \cite{Tan15} and the references therein).

\medskip
For the corresponding problem in the Hermitian case, Quillen \cite{Qui68} and Catlin-D'Angelo \cite{CD97} proved independently the following positivstellensatz: for any Hermitian bihomogeneous polynomial $f$ on $\mathbb C^n$ which is positive on $\mathbb C^n\setminus \{0\}$, there exists $\ell_o>0$ such that for any $\ell\geq\ell_o$,  there are homogeneous holomorphic polynomials $g_1,\cdots,g_N$ on $\mathbb C^n$ satisfying

\begin{equation}\label{polynomial}
\left(\sum_{i=1}^n |z_i|^2\right)^\ell\cdot f(z)=\sum_{j=1}^N |g_j(z)|^2,
\end{equation}
where $z=(z_1,\cdots,z_n)$.  Later, Catlin-D'Angelo generalized this positivstellensatz to the case of positive Hermitian algebraic functions on holomorphic line bundles over compact complex manifolds, and this was formulated as an isometric embedding theorem of the associated Hermitian metrics on the line bundles in \cite{CD99} (see Section \ref{Section 2} for the two equivalent definitions of Hermitian algebraic functions and Catlin-D'Angelo's result
stated as Theorem \ref{TheoremCD99}).
For holomorphic line bundles $L$ and $E$ over a compact complex manifold $X$ and  positive Hermitian algebraic functions $R$ and $P$ on $L$ and $E$ respectively such that
$R$ satisfies the strong global Cauchy-Schwarz (SGCS) condition, Catlin-D'Angelo obtained their positivstellensatz by proving that, if $m$ is a sufficiently large positive integer, then the associated integral operator ${\bf K}_{R^mP,\Omega}$ on  $H^0(X, L^m\otimes E)$ is positive (see (\ref{KRmP}) and Section \ref{Section 2} for the definitions of ${\bf K}_{R^mP,\Omega}$ and SGCS respectively).  Here $\Omega$ denotes the volume form on $X$ induced from $R$.  The arguments of Catlin-D'Angelo in \cite{CD99} depend on Catlin's result \cite{Cat99} about perturbations of the Bergman kernel on the unit disk bundle associated to a negative line bundle.  Later Varolin \cite{Var08} refined the result on the eventual positivity of ${\bf K}_{R^mP,\Omega}$ by using an  asymptotic expansion for the Bergman kernel for high powers of holomorphic line bundles obtained by Berman-Berndtsson-Sj\"ostrand \cite{BBS08}.

\medskip
In this article, we give an elementary and geometric proof of an asymptotic formula, which leads to the eventual positivity of the above integral operator.  We state our main result as follows:

 \begin{theo}\label{main-theorem}
Let $L $ and $E$ be  holomorphic line bundles over an $n$-dimensional compact complex manifold $X$.
Suppose $R$ and $P$ are positive Hermitian algebraic functions on $L$ and $E$ respectively,
    such that $R$ satisfies the strong global Cauchy-Schwarz condition.
Then there exists a constant $C>0$ such that for all $m\in \mathbb N$ and all $s\in H^0(X, L^m\otimes E)$, one has
\begin{equation}\label{maininequality}
\Modulus{  {\bf K}_{R^mP,\Omega}(s,s)    -\dfrac{\pi^n}{m^n}\norm{ s}^2}\leq \dfrac{C}{m^{n+1}}\norm{ s}^2.
\end{equation}
Here $\Omega$ denotes the volume form on $X$ induced from $R$, and $\norm{s}$ denotes the $L^2$ norm of $s$ with respect to $R^mP$ and $\Omega$.
\end{theo}

 \medskip
In comparison to the approaches of Catlin-D'Angelo \cite{CD99} and Varolin \cite{Var08} which depend on estimates on Bergman kernels, our proof of the above asymptotic formula is relatively elementary, direct and geometric, and does not depend on estimates on Bergman kernel.  Roughly speaking, our approach is to consider separately the behaviour of the kernels of the integral operators in tubular neighborhoods of the diagonal of $X\times X$ as well as that in the complementary regions, and derive our desired estimates by constructing good approximants of the kernels of the integral operators in the tubular neighborhoods of carefully chosen radii.

\medskip
We remark that while the two different approaches of \cite{Var08} and the present article both lead to the same dominant term
(i.e. $\dfrac{\pi^n}{m^n}\norm{s}^2$) in the approximation of ${\bf K}_{R^mP,\Omega}(s,s)  $,
it is not clear which of the two approaches will lead to a better higher order asymptotics for the
approximation (such as a better estimate for the constant $C$ in (\ref{maininequality})).  This appears to be an interesting question worthy of further investigation.

\medskip
A Hermitian algebraic function $Q$ on a holomorphic line bundle $F$ over a compact complex manifold
$X$ is called {\it a maximal sum of Hermitian squares} (resp. {\it a sum of Hermitian squares}) if there exists a
basis (resp. finite subset) $\{s_0,s_1\cdots,s_N\}$ of
$ H^0(X, F)$ such that one has
$Q(x,\overline{x})=\sum_{i=0}^N s_i(x)\overline{s_i(x)}$ for all $x\in X$.
Note that in this case, if $Q$ is positive, then as in \cite[Theorem 3]{CD99}, the associated map $\phi: X\to \mathbb P^{N}$ given by $\phi(x) = [s_0(x),\cdots, s_N(x)]$ is holomorphic and it induces an isometry between the Hermitian holomorphic line bundles $(F^*, h_Q)$ and $
(\mathcal O_{\mathbb P^{N}}(-1), h_N)$, i.e. $\phi^*\mathcal O_{\mathbb P^{N}}(-1)=F^*$ and $ \phi^*h_N=h_Q$.  Here  $h_N$
denotes the Hermitian metric on the universal line bundle $\mathcal O_{\mathbb P^{N}}(-1)$ over $\mathbb P^{N}$ induced by the polynomial $\sum_{i=0}^N |z_i|^2$, and $h_Q$ denotes the Hermitian metric on $F^*$ induced from $Q$.
As is known in Catlin-D'Angelo \cite{CD99} and Varolin \cite{Var08}, the eventual positivity of the integral operator in Theorem \ref{main-theorem} leads to the following positivstellensatz:

\begin{coro}\label{Corollary 1}  Let $X$, $L$, $E$, $R$, $P$, $n$, $C$ be as in
Theorem \ref{main-theorem}.  Then for each integer $m>\dfrac{C}{\pi^n}$ (resp. $m\geq \dfrac{C}{\pi^n}$),
the Hermitian algebraic function $R^mP$ is a maximal sum of Hermitian squares (resp. a sum of Hermitian squares), and in particular, there exists some holomorphic map $\phi: X\to \mathbb P^{N}$ (with $N$ depending on $m$) such that
	$((L^m\otimes E)^*,h_{R^mP } )= (\phi^*\mathcal{O}_{\mathbb P^{N}}(-1), \phi^*h_N)$.
\end{coro}

\medskip
The organization of this paper is as follows.
In Section \ref{Section 2}, we cover some background material and introduce some notations.
In Section \ref{Section 3}, we give the construction
of the approximants to the kernels of the integral operators being investigated.
In Sections \ref{Section 4}-\ref{Section 6}, we derive the desired estimates needed for the proof of Theorem \ref{main-theorem}.

\bigskip
\section{Notation and Background materials}\label{Section 2}

\medskip
In this section, we recall some background materials regarding Hermitian algebraic
functions on holomorphic line bundles, which are taken from \cite{CD99}, \cite{DV04} and \cite{Var08}.  As such, we will skip their proofs here and refer the reader to these references for their proofs.

\medskip
Let
$X$ be an $n$-dimensional compact complex manifold, and let $F$ be a holomorphic line bundle over $X$ with the corresponding projection map denoted by $\pi:F\to X$.    The dual holomorphic line bundle of $F$ is denoted by $F^*$.  The total space of $F^*$ (denoted by the same symbol) is a complex manifold, and the
complex conjugate manifold of $F^*$ (resp. $X$) is denoted by $\overline{ F^*}$ (resp. $\overline{X}$).  Following \cite{Var08}, a {\it Hermitian algebraic function on $F$} is a function $Q: F^*\times \overline{ F^*}\to\mathbb C$ such
that (i) $Q$ is holomorphic on $F^*\times \overline{ F^*}$, (ii)
$Q(v,\overline{w}) =\overline{ Q(w,\overline{v})}$ for all $v,w\in F^*$, and (iii) $Q(\lambda\cdot v,\overline{w}) =\lambda Q(v,\overline{w})$ for $\lambda\in\mathbb C$ and $v,w\in F^*$, where $\lambda\cdot v$ denotes scalar multiplication along fibers.
Next we consider the following alternative definition of $Q$.  Let $\rho_1:X\times \overline{X}\to X$ and $\rho_2:X\times\overline{X}\to \overline{X}$ denote the projection maps onto the first and second factor respectively, and consider the holomorphic line bundle $\rho_1^*F\otimes \rho_2^*\overline{F}$ over the complex manifold $X\times \overline{X}$, whose fiber at a point $(x,\overline{y})\in X\times \overline{X}$ is naturally isomorphic to $F_x\otimes \overline{F_y}$.  Here $F_x:=\pi^{-1}(x)$ denotes the fiber of $F$ at the point $x\in X$, etc.  Then with slight abuse of notation, $Q$ is alternatively defined as a global holomorphic section of $\rho_1^*F\otimes \rho_2^*\overline{F}$ (i.e., $Q\in H^0(X\otimes  \overline{X}, \rho_1^*F\otimes \rho_2^*\overline{F})$) satisfying the condition
$Q(x,\overline{y})=\overline{Q(y,\overline{x})}\in F_x\otimes\overline{F_y}$ for all $x,y
\in X$.
Under the above two alternative definitions, one has
\begin{equation}\label{equivalent}
Q(v,\overline{w})=\langle Q(x,\overline{y}), v\otimes\overline{w}\rangle \quad\text{for all }v\in F^*_x,~w\in F^*_y,~x,y\in X,
\end{equation}
where $\langle ~,~ \rangle$ denotes the pointwise pairing 
between $\rho_1^*F\otimes \rho_2^*\overline{F}$ and its dual line bundle induced from 
that between $F$ and $F^*$.  Note that there is no confusion on which definition of $Q$ is being used, 
as this is indicated by the object at which $Q$ evaluates.   We will often make our 
statement using only one of these two definitions, and leave the corresponding statement in terms 
of the other definition as an exercise to the reader.

\medskip
From the second definition above, one also sees that with respect to any basis $\{s^\alpha\}$ of $H^0(X, F)$,
there exists a corresponding Hermitian matrix $\big( C_{\alpha\overline{\beta}}\big)$ such that,
for all $x,y\in X$, one has
\begin{equation}\label{Pxy}
Q(x,y)=\sum_{\alpha,\beta}C_{\alpha\overline{\beta}}s^\alpha(x)\overline{s^\beta(y)}.
\end{equation}
 As mentioned in the last section, we say that $Q$ is {\it a sum of Hermitian squares} (resp. {\it a maximal sum of Hermitian squares}) if the Hermitian matrix
$\big(C_{\alpha\overline{\beta}}\big)$ with respect to one (and hence any) basis of $H^0(X, F)$ is positive semi-definite (resp. positive definite), or equivalently, there exists a finite subset (resp. a basis) $\{t^\alpha\}$ of $H^0(X, F)$ such that $Q(x,\overline{x})=\sum_\alpha t^\alpha(x)\overline{t^\alpha(x)}$ for all $x\in X$.
Also, the Hermitian algebraic function $Q$ is said to be {\it positive} if $Q(v,\overline{v})>0$ for all $0\neq v\in F^*$.  If $Q$ is positive, then $Q$ induces a Hermitian metric $h_Q$ on $F^*$ given by $h_Q(v,w)=Q(v,\overline{w})$ for $v,w\in F^*_x$, $x\in X$.  (In \cite{CD99}, such a Hermitian metric arising from a positive Hermitian algebraic function is called a {\it globalizable metric}.)  We recall that the curvature form $\Theta_{h_Q}$ of the Hermitian metric $h_Q$ is the $(1,1)$-form on $X$ given locally as follows:   On
any open subset $U$ of $X$ and for any local non-vanishing holomorphic section $s$ of $F^*\big|_U$, one has $\Theta_{h_Q}\big|_U=-\sqrt{-1}\partial\overline\partial\log h_Q(s,s)$.  Again we denote by $\pi:F^*\to X$ the projection map.  Following \cite{Var08} again (and with origin in \cite{CD99}), a positive Hermitian algebraic function $Q$ on $X$ is said to satisfy the {\it strong global Cauchy-Schwarz (SGCS) condition} if the following two conditions are satisfied:

\smallskip
(SGCS-1) $|Q(v,\overline{w})|^2<Q(v,v)Q(w,\overline w)$ for all non-zero $v,w\in F^*$ such that $\pi(v)\neq \pi(w)$.
(Note that one always has $|Q(v,\overline{w})|^2=Q(v,v)Q(w,\overline w)$ whenever $\pi(v)= \pi(w)$.)

\smallskip
(SGCS-2) The $(1,1)$-form $\Theta_{h_Q}$ on $X$ is negative definite.

\medskip
Let $\Omega$ be a smooth Hermitian volume form on $X$, and let $Q$ be a positive Hermitian algebraic function on $F$ as before.
We endow the vector space $H^{0}(X, F)$ with the $L^2$ Hermitian inner product  (induced from $Q$ and $\Omega$) given
as follows:  For $s,t\in H^{0}(X, F)$, one has
\begin{equation}\label{L2inner product}
(s,t):=\int_X \langle s,t\rangle_Q(x)\, \Omega(x),
  \quad\text{where }  \langle s,t\rangle_Q(x):= \dfrac{s(x)\overline{t(x)}}{Q(x,\overline x)}
\end{equation}
denotes the pointwise Hermitian pairing on $F$ dual to $(F^*,\,h_Q)$.
Note that the quotient in (\ref{L2inner product}) makes sense and is a scalar-valued function on $X$,
since both the numerator $s(x)\overline{t(x)}$ and the denominator $Q(x,\overline x)$ take values
in $F_x\otimes \overline{F_x}$.
For simplicity, we denote the associated $L^2$-norm (resp. pointwise norm) of $s$ by $\norm{s}$ (resp.
$\ptNorm{s}{x}$,~$x\in X$), i.e.,
\begin{equation}\label{pointwisenorm}\norm{s}=\sqrt{(s,s)}\quad\text{and}\quad
\ptNorm{s}{x}=\sqrt{ \langle s,s\rangle_Q(x)}.\end{equation}
Next one defines an integral operator ${\bf K}_{Q,\Omega}$ associated to $Q$ and $\Omega$ and acting (as a Hermitian bilinear form) on the vector space
$H^0(X, F)$ as follows:   For $s,t\in   H^0(X, F)$, we let
\begin{equation}\label{integraloperator}
{\bf K}_{Q,\Omega}(s,t):=\iint_{ X\times X}\dfrac{Q(x,\overline{y}){s(y)}\overline{t(x)}}{Q(x,\overline{x})Q(y,\overline{y})}\Omega(x)\Omega(y) .
\end{equation}
Note that as in (\ref{L2inner product}), the quotient in the integrand in (\ref{integraloperator}) makes sense as a scalar-valued function on $X\times X$.
With respect to an orthonormal basis $\{s^\alpha\}$ of $ H^0(X, F)$ for the Hermitian inner product in (\ref{L2inner product}), it is easy to see that the Hermitian matrix $\big(C_{\alpha\overline{\beta}}\big)$ associated to $Q$ as given in
(\ref{Pxy})  is simply given by $C_{\alpha\overline{\beta}}= {\bf K}_{Q,\Omega}(s^\alpha,s^\beta)$ for each $\alpha, \beta$.  It follows that $Q$ is a sum of Hermitian squares (resp. a maximal sum of Hermitian squares)
 if and only if the integral operator ${\bf K}_{Q,\Omega}$ is
positive semi-definite (resp. positive definite) in the sense that
${\bf K}_{Q,\Omega}(s,s)\geq 0$ (resp. ${\bf K}_{Q,\Omega}(s,s)>0$) for all $0\neq s\in  H^0(X, F)$.

\medskip
For a positive Hermitian algebraic function $Q$ on $X$, we define the {\it Cauchy-Schwarz function} $\schFn{Q}:X\times X\to \mathbb R$ associated to $Q$ given by
 \begin{equation}\label{PsiP}
 \schFn{Q}(x,y) := \frac{Q(\xy)Q(\yx)}{Q(\xx)Q(\yy)},\quad x,y\in X.
 \end{equation}
 From the positivity of $Q$ and as in (\ref{integraloperator}), one easily sees that $\schFn{Q}$ is a
 well-defined real-analytic function on $X\times X$.  Note also that $\Psi_Q(x,y)=\Psi_Q(y,x)$ for all $x,y\in X$.

\medskip
We remark that for two positive Hermitian algebraic functions $Q_1$ and $Q_2$ on two holomorphic line bundles $F_1$ and $F_2$ over $X$, the product $Q_1Q_2$ (obtained by taking pointwise multiplication (resp. tensor product) when the $Q_i$'s are taken as functions (resp. bundle-valued sections))  is a positive Hermitian algebraic function on $F_1\otimes F_2$.  Furthermore,
one easily sees that

\begin{equation}\label{PsiPP}
 \schFn{Q_1Q_2}(x,y)= \schFn{Q_1}(x,y)\cdot \schFn{Q_2}(x,y)\quad\text{for all }x,y\in X.
\end{equation}

\medskip
For the remainder of this section, we will fix two holomorphic line bundles $L $ and $E$ over $X$.  We also fix two positive Hermitian algebraic functions $R$ and $P$ on $L$ and $E$ respectively, such that $R$ satisfies the SGCS condition.  From (SGCS-2),  $X$ is endowed with a  K\"ahler form $\omega$ and an associated Hermitian volume form $\Omega$ given by
\begin{equation}\label{volumeform}\omega:=-\Theta_{h_R}\quad\text{and}\quad
\Omega:=\dfrac{\omega^n}{n!}.
\end{equation}

\medskip
We recall the following result of Catlin-D'Angelo:

\begin{theo}\label{TheoremCD99} \cite{CD99}
Let $L $ and $E$ be  holomorphic line bundles over an $n$-dimensional compact complex manifold $X$.
Suppose $R$ and $P$ are positive Hermitian algebraic functions on $L$ and $E$ respectively, such that $R$ satisfies the SGCS condition.
Then there exists $m_o\in \mathbb N$ such that for each integer
$m\geq m_o$, $R^mP$ is a maximal sum of Hermitian squares (on the line bundle $L^m\otimes E$); in particular, there exists some holomorphic map $\phi: X\to \mathbb P^{N}$ (with $N$ depending on $m$) such that
	$((L^m\otimes E)^*,h_{R^mP } )= (\phi^*\mathcal{O}_{\mathbb P^{N}}(-1), \phi^*h_N)$.
	Here $h_N$ is as in Corollary \ref{Corollary 1}.
	\end{theo}

We remark that  Catlin-D'Angelo obtained the above theorem by proving the positive-definiteness of the integral operators ${\bf K}_{R^mP,\Omega}$ for all sufficiently large $m$, where $\Omega$ is as in (\ref{volumeform}).

\bigskip
\section{ The integral operator and the approximant}\label{Section 3}

\medskip
In this section, we are going to construct approximants to the kernels of the integral operators in Theorem \ref{main-theorem}.
The Cauchy-Schwarz functions defined in (\ref{PsiP}) and the canonical coordinates (called Bochner coordinates in this article) associated to analytic K\"ahler metrics (as given in \cite{Boc47} and \cite{Cal53}) will play important roles in our construction.

\medskip
Throughout this section and as in Theorem \ref{main-theorem}, we let $L $ and $E$ be  holomorphic 
line bundles over an $n$-dimensional compact complex manifold $X$, and we let $R$ and $P$ be 
positive Hermitian algebraic functions on $L$ and $E$ respectively, such that $R$ satisfies the 
SGCS condition.   We recall from (\ref{volumeform}) the analytic K\"ahler form $\omega$ and the 
volume form $\Omega$ on $X$ induced from $R$.  Recall from Section \ref{Section 2} that for each $m\in \mathbb N$, $R^mP$ is a positive Hermitian algebraic function on the holomorphic line bundle $L^m \otimes E$, and one has an associated integral operator given by
\begin{equation}\label{KRmP}
{\bf K}_{R^mP,\Omega}(s,t) :=\iint_{ X\times X} \dfrac{R^m(\xy)P(\xy){s(y)}\overline{t(x)}}{R^m(\xx)P(\xx)R^m(\yy)P(\yy)}\Omega(x)\Omega(y)
\end{equation}
for $s,t\in   H^0(X, L^m \otimes E)$ (cf. (\ref{integraloperator})). We denote the diagonal of $X\times X$ by
\begin{equation} \label{diagonal}
D:=\{(x,x)\in X\times X\,\big|\, x\in X\}\cong X.
\end{equation}
Consider the real-analytic subvarieties of $X\times X$ given by
\begin{align}\label{ZR}
Z_R:&=\{(x,y)\in X\times X\,\big|\, R(\xy)=0\},\quad\text{and}\\
\label{ZP}
Z_P:&=\{(x,y)\in X\times X\,\big|\, P(\xy)=0\}.
\end{align}
From the positivity of $R$ and $P$, one easily sees that $D\cap(Z_R\cup Z_P)=\emptyset$.
Recall from \cite{Cal53}, p.{~}3, the {\it diastatic function} $\Phi$ associated to the analytic K\"ahler form $\omega$, which is defined on some open neighborhood of $D$ in $X\times X$.  In our present case where $\omega$ arises from $R$, $\Phi$ actually extends to a function on $(X\times X)\setminus Z_R$ given by
\begin{equation}\label{diastasis}
\Phi=-\log  \schFn{R}, \quad \text{where }\schFn{R}(x,y)=\frac{R(\xy)R(\yx)}{R(\xx)R(\yy)},\quad x,y\in X,
\end{equation}
is the Cauchy-Schwarz function associated to $R$ (cf. (\ref{PsiP})).  Later we will also need to consider
the Cauchy-Schwarz function $\schFn{P}$ associated to $P$ given by
\begin{equation}\label{PPsiP}
\schFn{P}(x,y)=\frac{P(\xy)P(\yx)}{P(\xx)P(\yy)},\quad x,y\in X.
\end{equation}
Recall also from \cite{Cal53}, p.14, and \cite{Boc47}, p. 181 that for any $x\in X$, there exists a canonical coordinate system $\hat z$ centered at $x$ given by an $n$-tuple of holomorphic
coordinate functions
 $z=(z_1,z_2,\cdots,z_n):B(x,r)\to\mathbb C^n$ such that $z(x)=0$ and the power series expansion of the diastatic function $\Phi(x,\cdot)$ takes the form
\begin{equation} \label{expanddiastasis}
\Phi(x,y)=|z(y)|^2+\sum_{|\alpha|,|\beta|\geq 2} \Phi_{\alpha\overline{\beta}}(\hat z)z(y)^\alpha\overline{z(y)^\beta},\quad
y\in B(x,r).
\end{equation}
Here $\alpha=(\alpha_1,\cdots,\alpha_n)$, $\beta=(\beta_1,\cdots,\beta_n)\in (\mathbb N\cup\{0\})^n$ are multi-indices, $|\alpha|=\alpha_1+\cdots+\alpha_n$,  $z^\alpha=z_1^{\alpha_1}z_2^{\alpha_2}\cdots z_n^{\alpha_n}$, $|z|^2=|z_1|^2+\cdots+|z_n|^2$, and the $ \Phi_{\alpha\overline{\beta}}(\hat z)$'s are Taylor coefficients of $\Phi$ at $x$ with respect to the coordinate system $\hat z$, etc; furthermore,
\begin{equation} \label{Bxr}
B(x,r):=\{y\in X\,\big|\, |z(y)|<r\}
\end{equation}
denotes the open coordinate ball (with respect to $\hat z$) centered at $x$ and of radius $r$.
Note that the right hand side of (\ref{expanddiastasis}) does not possess any monomial term in
$z^\alpha\overline{z^\beta}$ with $|\alpha|\leq 1$ or $|\beta|\leq1$, except when $|\alpha|=|\beta|=1$; in particular, it does not possess any monomial term of total degree $|\alpha|+|\beta|=3$.
For simplicity, any local coordinate system $\hat z$ (with associated coordinate functions $z$) satisfying (\ref{expanddiastasis})
will be called a {\it Bochner coordinate} at $x$.   It was shown in \cite{Cal53}, pp. 14-15, that
if $\hat z$ (with associated coordinate functions $z$) is a Bochner coordinate at $x$, then a
coordinate system $\hat z^\prime$ (with associated coordinate functions $z^\prime$) is a
Bochner coordinate at $x$ if and only if \begin{equation} \label{Uz}
z^\prime=Uz
\end{equation}
for some $U\in {\mathbb U}(n)$, where ${\mathbb U}(n)$ denotes the group of $n\times n$ unitary matrices.
In particular, the coordinate ball $B(x,r)$ in (\ref{Bxr}) is a well-defined open subset of $X$ independent of the choice of the Bochner coordinate $\hat z$ at $x$, and we will simply call it the {\it Bochner ball} centered at $x$ and of radius $r$.
Then one easily sees from the arguments in \cite{Boc47}, p. 181 on the
existence of Bochner coordinates that for any $x_o\in X$, any Bochner
coordinate $\hat z$ at $x_o$ (with associated coordinate functions $z:B(x_o,r)\to\mathbb C^n$),
there exist some $r^\prime$ satisfying $0<r^\prime<r$ and open subsets
$V\subset X$, $W \subset X\times X$,
such that $x_o\in V$,
\begin{equation} \label{W}
W=\bigcup_{x\in V}W_{x}, \quad\text{where } W_x:=\{x\}\times B(x,r^\prime)\cong B(x,r^\prime),
\end{equation}
and there exists a continuous function $\widetilde{z}:W\to \mathbb C^n$ such that
the restriction $\widetilde{z}\big|_{ W_x}:B(x, r^\prime)\to \mathbb C^n$ (under the identification in (\ref{W}))
gives a Bochner coordinate at $x$ for each $x \in V$, and $\widetilde{z}\big|_{ W_{x_o}}=z$ on $B(x_o,r^\prime)$; furthermore,
shrinking $r^\prime$ and $V$ if necessary, we may assume that for all $x\in V$ and all Bochner
coordinates $\hat z$ at $x$, the associated coordinate functions $z$ (which are necessarily of the form
$U\widetilde{z}\big|_{ W_x}$ for some $U\in\mathbb U(n)$ (cf. (\ref{Uz}))) are defined on $B(x,r^\prime)$.  Together with the compactness of $X$, it follows readily that there exists some constant $r_1>0$ such that for all $x\in X$ and all Bochner coordinates $\hat z$ at $x$, the associated coordinate functions are defined on $B(x,r_1)$.
Furthermore, the Bochner coordinates form a principal ${\mathbb U}(n)$-bundle
$p:\mathcal G\to X$ over $X$ such that for each $x\in X$ and $\hat z\in\mathcal G_x:= p^{-1}(x)$,
$\hat z$ is a Bochner coordinate at $x$ with associated coordinate functions $z:B(x,r_1)\to\mathbb C^n$,
and for each $U\in\mathbb U(n)$, $U\hat z$ is simply the Bochner coordinate at $x$ with associated
coordinate functions given by $Uz$.  Throughout this article, we will fix a choice of the constant $r_1$,
and same remark will
apply to the other constants $r_i$'s and $C_j$'s defined later, unless stated otherwise.

\medskip
Next we consider the power series expansions of the local expressions for $\schFn{R}$, $\schFn{P}$ and $\Omega$.  Take a point $x\in  X$ and a Bochner coordinate $\hat z$ at $x$ with associated holomorphic coordinate functions $z:B(x,r_1)\to\mathbb C^n$.  From (\ref{diastasis}) and upon exponentiating the negative of both sides of (\ref{expanddiastasis}), one easily sees that the power series expansion of $\schFn{R}$ with respect to $\hat z$ takes the form
\begin{equation}\label{expandPsiR}
\schFn{R}(x,y)=1-|z(y)|^2+\sum_{|\alpha|,|\beta|\geq 2}\Psi_{R,\alpha\overline{\beta}}(\hat z)z(y)^\alpha\overline{z(y)^\beta}, \quad y\in B(x,r_1).
\end{equation}
Here as in (\ref{expanddiastasis}), the $ \Psi_{R,\alpha\overline{\beta}}(\hat z)$'s are the Taylor coefficients of $\schFn{R}$ at
$x$ with respect to the coordinate system $\hat z$.  Next we write
\begin{equation}\label{omegay}
\omega=\dfrac{\sqrt{-1}}{2}\sum_{1\leq i,j\leq n}\omega_{i\overline{j}}(\hat z)\, \diff z_i\wedge \diff \overline{z_j}\quad\text{on } B(x,r_1).
\end{equation}
Here, for $1\le i,j\le n$, $\omega_{i\overline{j}}(\hat z)$ denotes the $(i,j)$-th component of $\omega$ with respect to $\hat z$.
From (\ref{volumeform}) and (\ref{diastasis}), one has $\omega_{i\overline{j}}(\hat z)=\partial_{z_i}\partial_{\overline{z_j}}\Phi(x,\cdot)$.
Together with (\ref{expanddiastasis}), one easily sees that the power series expansion of $\omega_{i\overline{j}}(\hat z)$ takes the form
\begin{equation}\label{omegaij}
\omega_{i\overline{j}}(\hat z)(y)
=\delta_{ij}+\sum_{|\alpha|,|\beta|\geq 1}
\omega_{i\overline{j},\alpha\overline{\beta}}(\hat z)z(y)^\alpha\overline{z(y)^\beta}
, \quad y\in B(x,r_1).
\end{equation}
Here $\delta_{ij}$ denotes the Kronecker symbol, and as before, the 
$\omega_{i\overline{j},\alpha\overline{\beta}}(\hat z)$'s denote the Taylor coefficients of $\omega_{i\overline{j}}(\hat z)$ at $x$ with respect to $\hat z$.
Next we
 let  $\diff V(\hat z)$ be the Euclidean volume form on $B(x,r_1)$ with respect to the Bochner coordinate $\hat z$ at $x$
 given by
\begin{equation}\label{dVz}
\diff V(\hat z) = \left(\frac{\sqrt{-1}}{2}\right)^n \diff z_1\wedge   \diff \overline{z_1} \wedge \cdots \wedge \diff z_n\wedge   \diff \overline{z_n},
\end{equation}
so that we may write
\begin{equation}\label{OmegaOmegaz}
\Omega = \Omega(\hat z)\,\diff V(\hat z) \quad\text{on } B(x,r_1),
\end{equation}
where  $ \Omega(\hat z)$ is a positive real-analytic function.  In fact, one easily sees that $\Omega(\hat z)=\det(\omega_{i\overline{j}}(\hat z))$ on $ B(x,r_1)$.  Together with (\ref{omegaij}), one easily sees that the power series expansion of $\Omega(\hat z)$ with respect to $\hat z$ takes the form
\begin{equation}\label{expandOmegaz}
 \Omega(\hat z)(y)=1+\sum_{|\alpha|,|\beta|\geq 1}\Omega_{\alpha\overline{\beta}}(\hat z)z(y)^\alpha\overline{z(y)^\beta}
, \quad y\in B(x,r_1).
\end{equation}
Here as before, the $\Omega_{\alpha\overline{\beta}}(\hat z)$'s are the Taylor coefficients of
$\Omega(\hat z)$ at $x$ with respect to $\hat z$.  As for $\schFn{P}$, we consider its power series expansion with respect to $\hat z$, which is valid on $B(x,r_2)$ for some $r_2$ satisfying $0<r_2<r_1$, so that we have
\begin{equation}\label{expandPsiP1}
\schFn{P}(x,y)=\sum_{\alpha,\beta}\Psi_{P,\alpha\overline{\beta}}(\hat z)z(y)^\alpha\overline{z(y)^\beta}, \quad y\in B(x,r_2).
\end{equation}
By differentiating the expression of $\schFn{P}$ in (\ref{PPsiP}), one easily sees that $\Psi_{P,\alpha\overline{\beta}}(\hat z)=1$ when $|\alpha|=|\beta|=0$, and $\Psi_{P,\alpha\overline{\beta}}(\hat z)=0$ when exactly one of the two numbers $|\alpha|,\, |\beta|$ is $0$.  Thus, one may refine (\ref{expandPsiP1}) as follows:
\begin{equation}\label{expandPsiP}
\schFn{P}(x,y)=1+\sum_{|\alpha|,|\beta|\geq 1}\Psi_{P,\alpha\overline{\beta}}(\hat z)z(y)^\alpha\overline{z(y)^\beta}, \quad y\in B(x,r_2).
\end{equation}
Furthermore, it is easy to see that shrinking $r_2$ if necessary, we may choose (and will choose) $r_2$ so that (\ref{expandPsiP}) holds for all $x\in X$ and all $\hat z\in  \mathcal G_x$.  Next we consider certain truncations of the power series expansions considered above.    Take $x\in  X$ and a Bochner coordinate $\hat z$ at $x$ with associated holomorphic coordinate functions $z:B(x,r_1)\to\mathbb C^n$ as before.  With notation as in (\ref{expandPsiR}), we define a function $\Psi_{R,\leq 4}(\hat z):B(x,r_1)\to \mathbb C$ given by
\begin{equation}\label{expandPsiR4}
\Psi_{R,\leq 4}(\hat z)(y):=1-|z(y)|^2+\sum_{|\alpha|=|\beta|=2}\Psi_{R,\alpha\overline{\beta}}(\hat z)z(y)^\alpha\overline{z(y)^\beta}, \quad y\in B(x,r_1).
\end{equation}
In other words, $\Psi_{R,\leq 4}(\hat z)$ is obtained by taking the sum of the monomial terms of total degree $|\alpha|+|\beta|\leq 4$ in the power series expansion of
$\schFn{R}(x,\cdot)$ at $x$ with respect to $\hat z$.  With notation as in (\ref{expandOmegaz}) and (\ref{expandPsiP}), we similarly define
\begin{align}\label{expandOmegaz2}
 \Omega_{\leq 2}(\hat z)(y)&:=1+\sum_{|\alpha|=|\beta|=1}\Omega_{\alpha\overline{\beta}}(\hat z)z(y)^\alpha\overline{z(y)^\beta}
, \quad y\in B(x,r_1);\\
\label{expandPsiP2}
\Psi_{P,\leq 2}(\hat z)(y)&:=1+\sum_{|\alpha|=|\beta|=1}\Psi_{P,\alpha\overline{\beta}}(\hat z)z(y)^\alpha\overline{z(y)^\beta}, \quad y\in B(x,r_2).
\end{align}

\begin{lemm}\label{sigma}
 There exists an open subset $X^\prime\subset X$ such that $\int_{X\setminus X^\prime}\Omega=0$ and the restriction
 $\mathcal G\big|_{X^\prime}$ admits a continuous section, i.e., there exists a continuous function $\sigma:X^\prime\to \mathcal G$ such that $\sigma(x)\in \mathcal G_x$ for all $x\in X^\prime$.
\end{lemm}

\begin{proof}  Since $X$ is compact, there exists a finite open cover $\{V_i \}_{1\leq i\leq N}$ of $X$ such that each $\mathcal G\big|_{V_i}$ is trivial, i.e., $\mathcal G\big|_{V_i}$ admits a continuous section $\sigma_i$ for each $i=1,\cdots, N$.  Now, we let $U_1:=V_1$, and for each $2\leq k\leq N$, we let $\displaystyle U_k:=V_k\setminus \cup_{1\leq i\leq k-1}\overline{V_i}$.  Finally, we let $\displaystyle X^\prime:= \cup_{1\leq i\leq N}{U_i}$, and let $\sigma:X^\prime\to
\mathcal G$ be given by $\sigma(x):=\sigma_i(x)$ for each $x\in U_i$, $i=1,\cdots, N$.  Then one easily checks that $X^\prime$ and $\sigma$ satisfy all the conditions in the lemma.
\end{proof}

For each $r$ satisfying $0<r<r_1$, we let
\begin{equation} \label{Wr}
W(r):=\{(x,y)\in X\times X\,\big|\, y\in B(x,r)\}
\end{equation}
(cf. (\ref{Bxr})), which is an open neighborhood of $D$ in $X\times X$.  For discussion in ensuing sections, we define a function $\rho:W(r_1)\to\mathbb R$ given by
\begin{equation} \label{rhoxy}
\rho(x,y):=|z(y)|, \quad (x,y)\in W(r_1),
\end{equation}
where $z=(z_1,\cdots,z_n)$ are the coordinate functions associated to a (and hence any) Bochner coordinate
$\hat z\in\mathcal G_x$  (cf. (\ref{Uz})).  It is easy to see that $\rho$ is continuous on $W(r_1)$, but in general, $\rho$ is not symmetric in $x$ and $y$, i.e., $\rho(x,y)\neq \rho(y,x)$.
Let $Z_R$ and $Z_P$ be as in (\ref{ZR}) and (\ref{ZP}).  Since $D\cap(Z_R\cup Z_P)=\emptyset$, it is easy to see that there exists
a constant $r_3$ satisfying $0<r_3<r_1$ such that
\begin{equation} \label{Wr2}
 W(r_3)\cap(Z_R\cup Z_P)=\emptyset.
 \end{equation}
We proceed to construct the approximant to the kernel associated to the integral operator in (\ref{KRmP}) (see (\ref{Tm}) below).
First we let
\begin{equation} \label{ro}
r_o:=\min\{r_1,r_2,r_3\}>0.
\end{equation}
From now on and as given in Lemma \ref{sigma}, we fix an open subset  $X^\prime\subset X$ such that
$\int_{X\setminus X^\prime}\Omega=0$ and a continuous section
$\sigma:X^\prime\to \mathcal G\big|_{X^\prime}$ so that $\sigma(x)\in\mathcal G_x$ for each $x\in X^\prime$.  For $r>0$, let
\begin{equation}\label{Wprime}
W^\prime(r):=(X^\prime \times X)\cap W(r)=\{(x,y)\in X^\prime\times X\,\big|\, y\in B(x,r)\}.
\end{equation}
For each $m\in \mathbb N$, we define a function
$T^{(m)}_\sigma: W^\prime(r_o)\to\mathbb C$ given by
\begin{equation}\label{Tm}
T^{(m)}_\sigma(x,y):=\dfrac{(\Psi_{R,\leq 4}(\sigma(x))(y))^m\cdot \Psi_{P,\leq 2}(\sigma(x))(y)\cdot  \Omega_{\leq 2}(\sigma(x))(y)}{ \Omega(\sigma(x))(y)}
\end{equation}
for  $(x,y)\in W^\prime(r_o)$
(cf. (\ref{OmegaOmegaz}), (\ref{expandPsiR4}), (\ref{expandOmegaz2}) and (\ref{expandPsiP2})).  We remark that it follows readily from the continuity of $\sigma$ and the associated continuous family of Bochner coordinates (cf. the construction of $\mathcal G$) that $T^{(m)}_\sigma$ is a continuous function on $W^\prime(r_o)$.
Note that for $(x,y)\in W(r_o)$, we have $\schFn{R}(\yx)\neq 0$ and $\schFn{P}(\yx)\neq 0$ (cf. (\ref{Wr2})), and thus we have
\begin{equation}\label{RmPss}
\dfrac{R^m(\xy)P(\xy)\sect(y)\conj{\sect(x)}}
{R^m(\xx)P(\xx)R^m(\yy)P(\yy)}=
\schFn{R}^m(x,y)\schFn{P}(x,y)\cdot\frac{\sect(y)\conj{\sect(x)}}{R^m(\yx)P(\yx)}
\end{equation}
for $(x,y)\in W(r_o)$
(cf. (\ref{diastasis}) and (\ref{PPsiP})).  Also, since $\int_{X\setminus X^\prime}\Omega=0$,
it follows that the value of the right hand side of (\ref{KRmP}) remains unchanged if we
replace the domain of integration there by $X^\prime\times X$ (in lieu of $X\times X$).
Together with (\ref{KRmP}) and (\ref{RmPss}), it follows that for each $m\in \mathbb{N}$, each $r > 0$
and each $s\in H^0(X, L^m\otimes E)$,
if $r < r_o$, then
\begin{equation}\label{IntegrationScheme}
  {\bf K}_{R^mP,\Omega}(s,s)    -\dfrac{\pi^n}{m^n} \norm{s}^2 = \mathrm{I} + \mathrm{II} + \mathrm{III},
\end{equation}
where
\begin{align}
\mathrm{I}\label{IIIIII}
   & := \iint_{(x,y)\in W^\prime(r)}\big(   \schFn{R}^m(x,y) \schFn{P}(x,y)- T^{(m)}_\sigma(x,y)\big)
    \frac{ \sect(y)\conj{\sect(x)}}{R^m(\yx)P(\yx)} \Omega(y)\Omega(x),\\
   \nonumber
\mathrm{II}& :=    \iint_{(x,y)\in W^\prime(r)}  T^{(m)}_\sigma(x,y)  \frac{ \sect(y)\conj{\sect(x)}}{R^m(\yx)P(\yx)} \Omega(y)\Omega(x)
                        - \frac{\pi^n}{m^n} \norm{\sect}^2,\\
\nonumber
\mathrm{III}
   & :=   \iint_{(x,y)\in( X^\prime\times X)\setminus W^\prime(r)}  \dfrac{R^m(\xy)P(\xy)\sect(y)\conj{\sect(x)}}
{R^m(\xx)P(\xx)R^m(\yy)P(\yy)}
             \Omega(y)\Omega(x).
\end{align}

\medskip
In the ensuing sections, we will make suitable choices of $r=r(m)$ for each $m\in\mathbb{N}$, which will allow us to obtain desired estimates for $\mathrm{I}$, $\mathrm{II}$ and $\mathrm{III}$.

\bigskip
\section{Estimation of I}\label{Section 4}

In this section, we are going to estimate the integral $\mathrm{I}$ in (\ref{IIIIII}).  First we have

\begin{lemm}\label{Taylorestimate}  There exist constants $C_1,\,C_2,\,C_3,\,C_4,\,C_5,\,C_6,\, r_4>0$ with $r_4<r_o$ such that one has
\begin{align}
\label{4.1}|\schFn{R}(x,y)-\Psi_{R,\leq 4}(\hat z)(y)|&\leq C_1\rho(x,y)^5,\\
\label{4.2}|\schFn{R}(x,y)-(1-\rho(x,y)^2)|&\leq C_2\rho(x,y)^4,\\
\label{4.3}|\schFn{P}(x,y)-\Psi_{P,\leq 2}(\hat z)(y)|&\leq C_3\rho(x,y)^3,\\
\label{4.4}|\schFn{P}(x,y)-1|&\leq C_4\rho(x,y)^2,\\
\label{4.5}|\Omega(\hat z)(y)-\Omega_{\leq 2}(\hat z)(y)|&\leq C_5\rho(x,y)^3,\quad\text{and}\\
\label{4.6}|\Omega(\hat z)(y)-1|&\leq C_6\rho(x,y)^2
\end{align}
 for all $(x,y)\in W(r_4)$ and all $\hat z\in\mathcal G_x$.  Here $\rho$ and $r_o$ is as in \eqref{rhoxy} and 
 \eqref{ro}.
\end{lemm}

\begin{proof}  To prove (\ref{4.1}), we take an arbitrary point $x_o\in X$, and
take an open subset $V$ of $X$ containing $x_o$ such that $\mathcal G\big|_V$ is trivial,
so that there exists a number $r_4^\prime>0$ and a continuous family of coordinate functions
$\{z_{\hat z}\}_{\hat z\in\mathcal G\big|_V}$ such that for each $x\in V$ and
each $\hat z\in\mathcal G_x$, $z_{\hat z} :B(x,r_4^\prime)\to\mathbb C^n$ are the
coordinate functions associated to $\hat z$.  Then $\schFn{R}$ gives rise to a continuous
family of real-analytic functions $\schFn{R}(\hat z):B(x,r_4^\prime)\to \mathbb R$ (and
given by a continuous family of power series expansions in the variables $z_{\hat z}$)
parametrized
by $\hat z\in \mathcal G\big|_{V}$ such that $\schFn{R}(\hat z)(y)=\schFn{R}(x,y)$  for
all $x\in V$, $y\in B(x,r_4^\prime)$ and $\hat z\in \mathcal G_x$.   By polarization, we obtain for each
$\schFn{R}(\hat z)$ a
holomorphic function  $\widetilde{\schFn{R}(\hat z)}$ on $B(x,r_4^\prime)\times \overline{B(x,r_4^\prime)} $
such that $\schFn{R}(\hat z)(y)=\widetilde{\schFn{R}(\hat z)}(y,\overline{y})$ for $y\in B(x,r_4^\prime)$.
Here  $ \overline{B(x,r_4^\prime)}$ denotes the complex conjugate manifold of $B(x,r_4^\prime)$.
Furthermore, it is clear that the $\widetilde{\schFn{R}(\hat z)}$'s form a continuous family of functions parametrized
by $\hat z\in \mathcal G\big|_{V}$.
Then it follows readily from standard theory for convergent power series of holomorphic functions (for the
$\widetilde{\schFn{R}(\hat z)}$'s)  that for some $C_1,r_4>0$ satisfying $0<r_4<r_4^\prime$,
(\ref{4.1}) holds for all $x\in V$, $y\in B(x,r_4)$ and 
$\hat z\in \mathcal G\big|_{V}$, upon shrinking $V$ if necessary.  Together with the compactness of $X$, it follows that (\ref{4.1}) holds
for all $(x,y)\in W(r_4)$ and all $\hat z\in\mathcal G_x$, upon shrinking $r_4$ and enlarging $C_1$ if necessary.
The proofs of (\ref{4.2}) to (\ref{4.6}) are the same as that of (\ref{4.1}), and thus they will be skipped.  We just remark that
the explicit expressions for the lower order terms in
(\ref{expandPsiR}), (\ref{expandOmegaz}), (\ref{expandPsiP}) are needed in the derivation of (\ref{4.2}), (\ref{4.4}) and (\ref{4.6}) respectively.
\end{proof}

Let $X^\prime$ and $\sigma:X^\prime\to \mathcal{G}$ be as chosen in Section \ref{Section 3}.
Let $m\in \mathbb{N}$,
	$r$ be a number satisfying $0 < r < r_o$,
	$s\in H^0(X, L^m\otimes E)$, and
	$\mathrm{I}$ be as in $\eqref{IIIIII}$.
Then one easily sees that
\begin{equation}\label{modulusI}
\Modulus{\mathrm{I}}
   \le \iint_{(x,y)\in W^\prime(r)} \Modulus{   \schFn{R}(x,y)^m \schFn{P}(x,y)- T^{(m)}_\sigma(x,y)}
 \cdot \Modulus{  \frac{ s(y)\overline{s(x)}}{R(\yx)^m P(\yx)} }\Omega(y)\Omega(x).
\end{equation}
By
\eqref{L2inner product}, \eqref{pointwisenorm} (with $Q$ there given by $R^mP$), \eqref{diastasis} and \eqref{PPsiP}
and using the identities $R(\yx) = \conj{R(\xy)}$, $P(\yx) = \conj{P(\xy)}$, one easily sees that
\begin{equation}
\Modulus{  \frac{ s(y)\overline{s(x)}}{R(\yx)^m P(\yx)} }^2=
\dfrac{\ptNorm{s}{x}^2\ptNorm{s}{y}^2}{ \schFn{R}(x,y)^m \schFn{P}(x,y) }.
\end{equation}
Together with \eqref{modulusI}, one has
\begin{multline} \label{tempEstOfONE}
\Modulus{\mathrm{I}}
   \le \iint_{(x,y)\in W^\prime(r)} \Modulus{ 1 -   \frac{T^{(m)}_\sigma(x,y)}{\schFn{R}(x,y)^m \schFn{P}(x,y)}}
                                                                                        \schFn{R}(x,y)^{\frac{m}{2}} \schFn{P}(x,y)^{\frac{1}{2}}
    \\ \cdot\ptNorm{s}{x}\cdot\ptNorm{s}{y}\,
  \Omega(y)\Omega(x).
\end{multline}
Next we consider a pointwise estimate for part of the integrand in \eqref{tempEstOfONE} as follows:
\begin{lemm} \label{PtwiseEstOfExprOne}
        There exist constants $C_7, \, r_5 > 0 $ with $ r_5 < r_4$ such that,
            for all $m \in \mathbb{N}$
            				and all $(x,y) \in W^\prime(\frac{r_5}{m^{\frac{1}{5}}})$,
 			one has
        \begin{equation} \label{PointwiseOfExprI}
        \Modulus{ 1 -   \frac{T^{(m)}_\sigma(x,y)}{\schFn{R}(x,y)^m \schFn{P}(x,y)}}
                \le C_7 (\rho(x,y)^3 + m\rho(x,y)^5).
        \end{equation}
Here $\rho(x,y)$ is as in \eqref{rhoxy}.
\end{lemm}

\begin{proof}
Recall from \eqref{Tm} that, for $(x,y)\in W^{\prime}(r_o)$,
\begin{equation}
\frac{T^{(m)}_\sigma(x,y)}{\schFn{R}(x,y)^m \schFn{P}(x,y)} = \mathrm{A}(x,y) \cdot \mathrm{B}(x,y)
\cdot \mathrm{C}(x,y)^m,\quad \text{where}
\end{equation}
\begin{equation}
\begin{aligned}
\mathrm{A} (x,y) &:= \frac{\Psi_{P,\le 2} (\sigma(x))(y)}{\Psi_P(x,y)} , \\
\mathrm{B} (x,y) &:= \frac{\Omega_{\le 2}(\sigma(x))(y)} { \Omega(\hat{z})(y)} , \\
\mathrm{C} (x,y) &:= \frac{\Psi_{R,\le 4}(\sigma(x))(y)}{\Psi_R(x,y)} .
\end{aligned}
\end{equation}
Using the identity $1 - \mathrm{A}  \mathrm{B}  \mathrm{C}^m
= (1- \mathrm{A}) + \mathrm{A} (1-\mathrm{B}) + \mathrm{A} \mathrm{B} \left( 1- \mathrm{C}^m\right)$, one has
\begin{multline} \label{RephrasePointwiseOfExprOne}
\Modulus{1 - \frac{T^{(m)}_\sigma(x,y)}{\schFn{R}(x,y)^m \schFn{P}(x,y)}}
\le\Modulus{1 - \mathrm{A}(x,y)} \\ + \Modulus{\mathrm{A}(x,y)} \Modulus{1- \mathrm{B}(x,y)}
            + \Modulus{\mathrm{A}(x,y)}\Modulus{\mathrm{B}(x,y)}\Modulus{1 - \mathrm{C}(x,y)^m}.
\end{multline}
Let $C_4$ be as in \eqref{4.4}.
Now we choose $r_5 < r_4$ so that $1- C_4r_5^2 > \frac{1}{2}$.
Then by \eqref{4.3} and \eqref{4.4}, for all $(x,y) \in W^\prime(r_5)$, one has
\begin{equation} \label{EstOfTruncSchFnP}
\Modulus{1- \mathrm{A}(x,y)} \le \frac{C_3 \rho(x,y)^3}{1 - C_4\rho(x,y)^2} \le\ \frac{C_3\rho(x,y)^3}{1 - C_4 r_5^2} \le C_8 \rho(x,y)^3,
\end{equation}
where $C_8 = 2C_3$.  Similarly, using \eqref{4.1}, \eqref{4.2}, \eqref{4.5}, \eqref{4.6}, and shrinking $r_5$ if necessary,
    one easily sees that there exist constants $C_9, C_{10} > 0$
        such that for all $(x,y)\in W^{\prime}(r_5)$, one has
\begin{align}
\Modulus{1- \mathrm{B}(x,y)} &\le C_9 \rho(x,y)^3 , \label{EstOfTruncOmega}\quad\text{and} \\
\Modulus{1-\mathrm{C}(x,y)} & \le C_{10} \rho(x,y)^5. \label{EstOfTruncSchFnR}
\end{align}
By \eqref{4.3} and \eqref{4.4}, and shrinking $r_5$ further if necessary,
    it is also clear that there exist constants $C_{11}, C_{12} > 0 $ such that
        for all $(x,y) \in W^{\prime}(r_5)$, one has
 \begin{align}
        \Modulus{\mathrm{A}(x,y)} &\le C_{11} \quad \text{and} \label{BdedTrunSchP}\\
        \Modulus{\mathrm{B}(x,y)} &\le C_{12}. \label{BdedTruncOmega}
        \end{align}
Now let $m\in \mathbb{N}$ and $(x,y)\in W^{\prime}(\frac{r_5}{m^{\frac{1}{5}}})$ be given.
By \eqref{EstOfTruncSchFnR}, one has
$
\Modulus{\mathrm{C}(x,y)} \le 1+\frac{C_{10} r_5^5}{m}
$,
so that for each $1\leq k\leq m-1$,
\begin{equation} \label{BdedTruncSchRPower}
\Modulus{\mathrm{C}(x,y)}^k  \le
 \left(1 + \frac{C_{10} r_5^5}{m}\right)^k \le \left(1 + \frac{C_{10} r_5^5}{m}\right)^m \le e^{C_{10} r_5^5},
 \quad\text{and thus}
\end{equation}
\begin{align} \label{EstOfTruncSchFnRPower}
& \Modulus{1 - \mathrm{C}(x,y) ^m} \\
\nonumber
= {}& \Modulus{1- \mathrm{C}(x,y)} \cdot \Modulus{1 + \mathrm{C}(x,y) + \mathrm{C}(x,y)^2 +
 \cdots + \mathrm{C}(x,y)^{m-1}} \\
\nonumber\le {}& C_{10} \rho(x,y)^5 \cdot m \cdot e^{C_{10} r_5^5} \\
\nonumber = {} & C_{13} \cdot m \rho(x,y)^5,\quad\text{where }C_{13} := C_{10} e^{C_{10} r_5^5}.
\end{align}
Combining \eqref{RephrasePointwiseOfExprOne},
                \eqref{EstOfTruncSchFnP},
                \eqref{EstOfTruncOmega},
                \eqref{BdedTrunSchP},
                \eqref{BdedTruncOmega},
                \eqref{EstOfTruncSchFnRPower},
        one has, for all $(x,y) \in W^\prime(\frac{r_5}{m^{\frac{1}{5}}})$,
        \begin{equation}
\begin{aligned}[t]
&\Modulus{ 1 -   \frac{T^{(m)}_\sigma(x,y)}{\schFn{R}(x,y)^m \schFn{P}(x,y)}} \\
\le {} &  C_8 \rho(x,y)^3 + C_{11} C_9 \rho(x,y)^3  + C_{11} C_{12} C_{13} m\rho(x,y)^5 \\
= {} &  C_{7} (\rho(x,y)^3 + m\rho(x,y)^5),
\end{aligned}
\end{equation}
where $C_{7} := \max \{C_{8} + C_{11}C_9, C_{11}C_{12}C_{13}\}$.
This finishes the proof of the lemma.
        \end{proof}
    \begin{lemm} \label{AsymEstOfRho}
    There exists a constant $ r_6 > 0$ with $r_6 < r_4$ such that,
            for all $(x,y)\in W(r_6)$, the quantity $\rho(y,x)$ is well-defined and $\rho(y,x) \le 2 \rho(x,y)$.
    \end{lemm}
    \begin{proof}
    First it follows readily from the compactness of $X$, the construction of $\mathcal{G}$
        and the definition of $\rho$ that there exists a constant $r_7 >0$ with $r_7 < r_4$ such that,
            for $(x,y) \in W(r_7)$, the quantity $\rho(y,x)$ is well-defined and satisfies   $\rho(y,x) < r_4$.
    By shrinking $r_7$ if necessary and using \eqref{4.2}, we may assume that
    \begin{equation} \label{AsymEstOfRhoPrep}
    \frac{1}{2}\rho(x,y)^2 \le \Modulus{\Psi_R(x,y) - 1} \le 2\rho(x,y)^2\quad\text{for all }(x,y) \in W(r_7).
    \end{equation}
    Repeating the above argument (with $r_4$ replaced by $r_7$),
    one sees that there exists a constant $r_6 > 0$ with $r_6 < r_7$ such that,
        for $(x,y) \in W(r_6)$, one has $\rho(y,x) < r_7$.
Now let $(x,y) \in W(r_6) $ be given (so that $(y,x)\in W(r_7)$).
Then using \eqref{AsymEstOfRhoPrep} but with the roles of $x$ and $y$ interchanged,
    one has
    \begin{equation} \label{AsymOfRhoEstRight}
    \frac{1}{2} \rho(y,x)^2 \le \Modulus{\Psi_R(y,x) - 1 }.
    \end{equation}
Together with the identity $\Psi_R(x,y) = \Psi_R(y,x)$ and the second inequality in \eqref{AsymEstOfRhoPrep},
        one has $\frac{1}{2}\rho(y,x)^2 \le 2\rho(x,y)^2$ and thus $\rho(y,x) \le 2\rho(x,y)$.
    \end{proof}

    For $\mathbb C^n$, we denote its Euclidean ball centered at $0$ and of radius $r$ and its Euclidean volume form by
    \begin{align}
    \label{EuclideanBall}B(r):&=\{z\in \mathbb C^n\,\big|\, |z|<r\}\quad\text{and}\\
    \label{Euclideanvolumeform}
    \diff V (z) &= \left(\frac{\sqrt{-1}}{2}\right)^n \diff z_1 \wedge \diff \conj{z_1} \wedge
                                \cdots \wedge \diff z_n \wedge \diff \conj{z_n}.
    \end{align}

    \begin{lemm}\label{intrhok}
    There exist constants $C_{14}, \, r_8 > 0$ with $r_8 < r_4$ such that,
        for all $k,r\geq 0$ and all $x\in X$,
            if $r < r_8$, then
            \begin{equation}
            \int_{y\in B(x,r)} \rho(x,y)^k \Omega(y) \le C_{14} r^{2n + k}.
            \end{equation}
    \end{lemm}
    \begin{proof}
    One easily sees from \eqref{4.6} that
        there exist constants $C_{15}, \, r_8 > 0$ with $r_8 < r_4$ such that,
         for all $x\in X$, all $\hat{z} \in \mathcal{G}_x$ and all $y\in B(x, r_8)$,
            one has $\Omega(\hat{z})(y) \le C_{15}$.
    Hence, in terms of the coordinate functions $z: B(x,r_4) \to \C^n$ associated to
    $\hat{z}$,
    \eqref{EuclideanBall} and \eqref{Euclideanvolumeform},
                    one has, for each $k>0$ and $0 < r< r_8$,
                    \begin{align}\label{rhok}
                     \int_{y\in B(x,r)} \rho(x,y)^k \Omega(y)
                        &\le C_{15} \int_{z\in B(r)} |z|^k \, \diff V (z)  \\
                  \nonumber       &= C_{15} \frac{2\pi^n r^{2n+k}}{(n-1)! (2n+k)} \\
                         &\le C_{14} r^{2n+k}, \quad\text{where }C_{14 } := 2\pi^n C_{15}.\nonumber
                    \end{align}
     Here the second line in \eqref{rhok} follows from a straightforward computation.
    \end{proof}

\begin{prop} \label{EstimateI}
There exist constants $C_{16}, \, r_{9} > 0$ such that,
    for all $m\in \mathbb{N}$,
            all $s\in H^0(X, L^m\otimes E)$
            and all $r$ satisfying $0 < r \le \frac{r_9}{m^{\frac{n+2}{2n+5}}} $, one has
            \begin{equation}
        \Modulus{\mathrm{I}} \le  \frac{C_{16}}{m^{n+1}} \norm{s}^2.
        \end{equation}
Here $\mathrm{I}$ is as in \eqref{IIIIII} (with $r$ there as above).
\end{prop}
\begin{proof}
Recall from the (SGCS-1) condition for $R$ that $|\Psi_R(x,y)|\le 1$ for all $(x,y) \in X\times X$.
One also easily sees from \eqref{4.4} that there exist constants $C_{17}, \, r_{10} > 0$ such that,
	for all $(x,y)\in W(r_o)$, one has $\Modulus{\Psi_P(x,y)} \le C_{17} $.
Now, we let $r_9 := \min \{r_0, r_4, r_5, \frac{r_8}{2}, r_{10}\} > 0$,
            where the $r_i$'s (and $C_j$'s) are as chosen before.
Let
$\mathrm{I}$ be as in \eqref{IIIIII} with $r$ there satisfying
$0<r\le \frac{r_9}{m^{\frac{n+2}{2n+5}}}$.  Note that for $n\geq 1$, one has $ \dfrac{1}{m^{\frac{n+2}{2n+5}}} <\frac{1}{m^{\frac{1}{5}}} $.
Hence from \eqref{tempEstOfONE} and Lemma \ref{PtwiseEstOfExprOne}, one has
\begin{equation} \label{prop1eq2}
\begin{aligned}[t]
\Modulus{\mathrm{I}}
    &\le C_7 C_{17} \iint_{(x,y)\in W^\prime(r)} \left(\rho(x,y)^3 + m\rho(x,y)^5\right) \ptNorm{s}{x}\,\ptNorm{s}{y} \Omega(y) \Omega(x)
            \\
   &= C_{7} C_{17}(\mathrm{J}_3(r) + m\mathrm{J}_5(r)),
\end{aligned}
\end{equation}
where, for $k > 0$ and $0 < r < r_0$,
\begin{equation}
\mathrm{J}_k (r) := \iint_{(x,y)\in W(r)} \rho(x,y)^k
    \ptNorm{s}{x}\,\ptNorm{s}{y} \Omega(y)\Omega(x).
\end{equation}
By the Cauchy-Schwarz inequality, one has
\begin{equation} \label{prop1eq4}
\Modulus{\mathrm{J}_k(r)} \le \left(\mathrm{J}_{k,1}(r)\right)^{\frac{1}{2}}
\left(\mathrm{J}_{k,2}(r)\right)^{\frac{1}{2}},
\end{equation}
where, as iterated integrals,
\begin{align}
\mathrm{J}_{k,1}(r) &:= \int_{x\in X} \int_{y\in B(x,r)} \rho(x,y)^{2k} \ptNorm{s}{x}^2
\Omega(y)\Omega(x), \label{prop1eq5} \\
\mathrm{J}_{k,2}(r) &:= \int_{x\in X} \int_{y\in B(x,r)} \ptNorm{s}{y}^2
                                     \Omega(y)\Omega(x). \label{prop1eq6}
\end{align}
By Lemma \ref{intrhok}, one has
\begin{equation} \label{prop1eq7}
\begin{aligned}[t]
\mathrm{J}_{k,1}(r) &= \int_{x\in X} \ptNorm{s}{x}^2 \int_{y\in B(x,r)} \rho(x,y)^{2k} \Omega(y)\Omega(x) \\
&\le C_{14} r^{2n+2k} \norm{s}^2.
\end{aligned}
\end{equation}
For each $y\in X$, let $B^\prime(y,r) := \{x\in X | \, \rho(x,y) < r\}$.
Note that, for $r \leq r_9$, it follows from Lemma \ref{AsymEstOfRho}
	that $B^\prime(y,r)\subseteq B(y,2r)\subseteq B(y,r_8)$.
Thus, upon interchanging the order of integration in \eqref{prop1eq6}, one has
\begin{equation} \label{propr1eq8}
\begin{aligned}[t]
\mathrm{J}_{k,2}(r) &= \int_{y\in X} \ptNorm{s}{y}^2 \int_{x\in B^\prime(y,r)} \Omega(x)\Omega(y) \\
&\le \int_{y\in X} \ptNorm{s}{y}^2 \int_{x\in B(y,2r)} \Omega(x)\Omega(y) \\
&\le C_{14} (2r)^{2n}  \norm{s}^2,
\end{aligned}
\end{equation}
where the last line follows from Lemma \ref{intrhok}.
Combining \eqref{prop1eq4}, \eqref{prop1eq7} and \eqref{propr1eq8},
    it follows that for $r \leq r_9$,
    \begin{equation} \label{prop1eq9}
    \begin{aligned}[t]
    \Modulus{\mathrm{J}_k(r)} &\le \left(C_{14} r^{2n+2k}\norm{s}^2\right)^{\frac{1}{2}} \left(C_{14} (2r)^{2n}\norm{s}^2\right)^{\frac{1}{2}} \\
    &= C_{14} 2^n r^{2n+k} \norm{s}^2.
    \end{aligned}
    \end{equation}
From \eqref{prop1eq2} and \eqref{prop1eq9}, one has
\begin{align}
\Modulus{\mathrm{I}} &\le C_7 C_{17} \left(C_{14} 2^n r^{2n+3} + m C_{14} 2^n r^{2n+5}\right) \norm{s}^2 \\
      \nonumber                            &= C_{18} \left(r^{2n+3} + m r^{2n+5}\right) \norm{s}^2,\quad\text{where }
      C_{18} := 2^n C_7 C_{17} C_{14}.
\end{align}
Now for each $r$ satisfying $0 < r \le \frac{r_9}{m^{\frac{n+2}{2n+5}}}$, one easily checks that
$
r^{2n+3} \le \frac{r_9^{2n+3}}{m^{n+1}}$ and $m r^{2n+5} \le \frac{r_9^{2n+5}}{m^{n+1}}$, and hence one has
\begin{equation}
|\mathrm{I}|\le  \frac{C_{16}}{m^{n+1}} \norm{s}^2,\quad\text{where }C_{16} : = C_{18} (r_9^{2n+3} + r_9^{2n+5}). \qedhere
\end{equation}
\end{proof}

\bigskip
\section{Estimation of II}\label{Section 5}

\medskip
In this section, we are going to estimate the expression $\mathrm{II}$ in (\ref{IIIIII}).
For $r>0$, let
$B(r)$ be as in \eqref{EuclideanBall}, and denote its closure by
$\overline{B(r)}:=\{z\in \mathbb C^n\,\big|\, |z|\leq r\}$.
Let $q$ be an analytic function admitting a power series expansion
$q(z) = \sum_{\alpha,\beta}  q_{\alpha\conj{\beta}} z^\alpha \conj{z^\beta}$ on $\overline{B(r)}$
(here the notation is as in (\ref{expanddiastasis})).
Then $q$
is said to {\emph{have only quasi-diagonal terms}} if $q_{\alphabeta} = 0$ whenever $|\alpha|\neq |\beta|$.

\begin{lemm} \label{meanValueProperty}
Let $f$ a holomorphic function admitting a power series expansion
$f(z)=\sum_{\alpha}f_\alpha z^\alpha$ on $\overline{B(r)}$, and let
 $q$ be an analytic function admitting a power series  expansion
 $ q(z) = \sum_{\alpha,\beta}  q_{\alpha\conj{\beta}} z^\alpha \conj{z^\beta}$ on $\overline{B(r)}$.   If $q$ has only quasi-diagonal terms, then
  \begin{equation}\label{ortho}
                    \integration{z}{B(r)}{\diff V (z)}{f(z) q(z)} = f(0) \integration{z}{B(r)}{\diff V (z)}{ q(z)}.
 \end{equation}
Here $\diff V (z)$ is as in \eqref{Euclideanvolumeform}.
\end{lemm}

\begin{proof}
First we recall that for multi-indices $\alpha$ and $\beta$, one has
\begin{equation} \label{orthogonalityOfMonomials}
 \int_{B(r)} z^\alpha \conj{z}^\beta \diff V(z) = 0 \quad\text{whenever }\alpha \neq \beta,
\end{equation}
which can be verified easily by considering the change of variables given by $(z_1,\cdots,z_n)\to (e^{i\theta_1}z_1,
\cdots,e^{i\theta_n}z_n)$, and then letting $\theta_1,\cdots,\theta_n$ vary.
Let $g$ be the function given by $g(z)=f(z) - f(0)$ for $z\in\overline{B(r)}$.
Then $g(0) = 0$ and $g$ is also a holomorphic function admitting a power series expansion $g(z) =
 \sum_{|\gamma| > 0} g_\gamma z^\gamma$
 on $\overline{B(r)}$, noting that $g_\gamma=g(0)=0$ when $|\gamma|=0$.
Then one has
\begin{align}\label{integralequality}
\int_{B(r)} g(z) q(z) \diff V(z)
&=   \sum_{ |\gamma| > 0} \sum_{\alpha,\beta} g_\gamma q_{\alphabeta}
          \int_{B(r)} z^{\alpha + \gamma} \conj{z}^\beta \diff V(z)\\
\nonumber    &= \sum_{ |\gamma| > 0} \sum_{\alpha} g_\gamma q_{\alpha, \conj{\alpha + \gamma}}
                                                                                                   \int_{B(r)} |z^{\alpha + \gamma}|^{2} \diff V(z),
\end{align}
where the last equality follows from (\ref{orthogonalityOfMonomials}).  Since $q$ has only quasi-diagonal terms, it follows that for each $\alpha$ and $\gamma$ satisfying $|\gamma|>0$, one has $|\alpha+\gamma|=|\alpha|+|\gamma|>|\alpha|$, and thus $q_{\alpha, \conj{\alpha + \gamma}}=0$.
Hence one has
$ \integration{z}{B(r)}{\diff V (z)}{g(z) q(z)} = 0,
$
which leads to (\ref{ortho}) readily.
\end{proof}

\begin{lemm} \label{explicitSpherIntegral}
Notation as in \eqref{EuclideanBall} and \eqref{Euclideanvolumeform}.
For each integer $m,k \ge 0$ and real number $a>0$,
\begin{equation}
\int_{z\in B\left(\frac{1}{\sqrt{a}}\right)} \Modulus{z}^{2k} (1 - a |z|^2) ^m \, \diff V(z)
	= \frac{\pi^n (n+k-1)! m!}{(n-1)! (m+k+n)! a^{n+k}}.
\end{equation}	
\end{lemm}

\begin{proof}
We will skip the proof, which follows from a direct calculuation.
\end{proof}

Let $X^\prime$ and $\sigma:X^\prime\to \mathcal{G}$ be as chosen in Section \ref{Section 3}.
Let $m\in \mathbb{N}$,
		$r$ be a number satisfying $0 < r < r_o$,
		$s\in H^0(X, L^m\otimes E)$, and
		$\mathrm{II}$ be as in $\eqref{IIIIII}$.
Rewriting the first term of $\mathrm{II}$ as an iterated integral, one has
\begin{equation} \label{exprTwoEqOne}
\mathrm{II} = \int_{x\in X^\prime} \Lambda(x) \Omega(x) - \frac{\pi^n}{m^n} \norm{s}^2,
\end{equation}
 where, for each $x\in X^\prime$,
 \begin{equation} \label{exprTwoEqTwo}
\Lambda(x) := \int_{y\in B(x,r)} T_\sigma^{(m)}(x,y) \frac{s(y)\conj{s(x)}}{R(\yx)^m P(\yx)} \Omega(y)
\end{equation}
Using \eqref{dVz}, \eqref{OmegaOmegaz}, \eqref{Tm},
	for $x\in X^\prime$, one has
\begin{equation}  \label{exprTwoEqThree}
\Lambda(x) = \int_{y\in B(x,r)} \lambda(x,y) \frac{s(y)\conj{s(x)}}{R(\yx)^m P(\yx)} \, \diff V(\sigma(x)) (y),
	\end{equation}	
where, for $(x,y)\in W^{\prime}(r_o)$,
\begin{equation} \label{exprTwoEqFour}
\lambda(x,y) := \Psi_{P, \le 2} (\sigma(x))(y) \cdot \Omega_{\le 2} (\sigma(x))(y) \cdot
					\left(\Psi_{R,\le 4}(\sigma(x))(y)\right)^m .
	\end{equation}	
For each $x\in X^\prime$
	 and in terms of the coordinate functions $z: B(x,r_0) \to \C^n$ associated to $\sigma(x)$,
	 	it follows readily from \eqref{expandPsiR4}, \eqref{expandOmegaz2}, \eqref{expandPsiP2}
	 		that $\lambda(x,y)$ is an analytic function in the variable $z=z(y)$ and has only quasi-diagonal terms.
Also, the quotient $\frac{s(y)\conj{s(x)}}{R(\yx)^m P(\yx)}$ is a holomorphic function in the variable $y$.
Thus, by Lemma \ref{meanValueProperty}, one has, for $x\in X^\prime$,
\begin{equation} \label{exprTwoEqSix}
\Lambda(x) = \ptNorm{s}{x}^2 \int_{y\in B(x,r)} \lambda(x,y) \, \diff V(\sigma(x))(y).
	 			\end{equation}	 			

Similar to Lemma \ref{PtwiseEstOfExprOne}, we have the following pointwise estimate:
\begin{lemm} \label{PtwiseEstOfExprTwo}
There exist constants $C_{19}, \, r_{10} > 0$ with $r_{10} < r_4$ such that,
	for all $m\in \mathbb{N}$ and all $(x,y)\in W^\prime(r_{10})$, one has
	\begin{equation} \label{exprTwoEqFive}
\begin{aligned} [t]
&\Modulus{\lambda(x,y) - \left(1 - \rho(x,y)^2\right)^m} \\
	\le {} &C_{19} \left(\rho(x,y)^2 + m\rho(x,y)^4\right) \left(1 - \frac{\rho(x,y)^2}{2}\right)^{m-1}.
	\end{aligned}
	\end{equation}	

\end{lemm}

\begin{proof}
For $(x,y)\in W^{\prime}(r_0)$, we rewrite \eqref{exprTwoEqFour} as
\begin{align} \label{exprTwoEqSix1}
\lambda(x,y) &= \mathcal{A}(x,y) \mathcal{B}(x,y)^m,\quad\text{where}\\
\mathcal{A}(x,y) &:= \Psi_{P, \le 2}(\sigma(x))(y) \cdot \Omega_{\le 2}(\sigma(x))(y),\notag \\
\mathcal{B}(x,y) &:= \Psi_{R,\le 4}(\sigma(x))(y).\notag
\end{align}
It follows readily from Lemma \ref{Taylorestimate} that
	there exist constants $C_{20}, \, C_{21}, \, r_{10} > 0$ with $r_{10} < \min \{r_4, 1\}$ such that,
		for all $(x,y)\in W^\prime(r_{10})$, one has
\begin{align}
\Modulus{\mathcal{A}(x,y) - 1} &\le C_{20} \rho(x,y)^2, \label{exprTwoEqSeven} \\
\Modulus{\mathcal{B}(x,y) - \left(1 - \rho(x,y)^2\right)}& \le C_{21} \rho(x,y)^4, \text{ and }\label{exprTwoEqEight} \\
0\le \mathcal{B}(x,y) &\le 1 - \frac{\rho(x,y)^2}{2}. \label{exprTwoEqNine}
		 \end{align}		
		 (For example, \eqref{exprTwoEqEight} follows immediately from \eqref{4.1} and \eqref{4.2},
		 	while \eqref{exprTwoEqNine} follows from \eqref{exprTwoEqEight},
		 		upon shrinking $r_{10}$ if necessary.)
By \eqref{exprTwoEqSix1}, one has
\begin{equation} \label{exprTwoEqTen}
\begin{aligned}[t]
&\lambda(x,y) - \left(1 - \rho(x,y)^2\right)^m \\
= {} & \left(\mathcal{A}(x,y) - 1\right) \mathcal{B}(x,y)^m + \mathcal{B}(x,y)^m - \left(1 - \rho(x,y)^2\right)^m \\
= {} & \left(\mathcal{A}(x,y) - 1\right) \mathcal{B}(x,y)^m \\
	& \qquad+
			\left[\mathcal{B}(x,y) - (1- \rho(x,y)^2)\right]
				\cdot \sum_{j=0}^{m-1} \mathcal{B}(x,y)^j \left(1 - \rho(x,y)^2\right)^{m-1 - j}.
\end{aligned}
		 		\end{equation}
For all $(x,y) \in W^{\prime}(r_{10})$, from \eqref{exprTwoEqNine}, one has
\begin{equation}
\begin{aligned}[t]
&\sum_{j=0}^{m-1} \mathcal{B}(x,y)^j \left(1 - \rho(x,y)^2\right)^{m-1 - j} \\
	\le {}&\sum_{j=0}^{m-1} \left(1 - \frac{\rho(x,y)^2}{2}\right)^j \left(1 - \rho(x,y)^2\right)^{m-1-j} \\
	\le {} &m \left(1 -\frac{\rho(x,y)^2}{2}\right)^{m-1},
		 				 		\end{aligned}		 	
\end{equation}
and together with \eqref{exprTwoEqSeven}, \eqref{exprTwoEqNine}, \eqref{exprTwoEqTen},
	one has
	\begin{align}			 					 		
	&\Modulus{\lambda(x,y) - \left(1 - \rho(x,y)^2\right)^m} \\
	\le {} & C_{20} \rho(x,y)^2 \left(1 - \frac{\rho(x,y)^2}{2}\right)^m
			 + C_{21}\rho(x,y)^4 m \left(1 - \frac{\rho(x,y)^2}{2}\right)^{m-1} \notag\\
	\le {} & C_{19} \left(\rho(x,y)^2 + m\rho(x,y)^4\right)\left(1 -\frac{\rho(x,y)^2}{2}\right)^{m-1}	,	 \notag
	\end{align}
	where $C_{19} := \max\{C_{20}, C_{21}\}$, noting that $0 < 1- \frac{\rho(x,y)^2}{2}< 1$.
\end{proof}

\begin{lemm} \label{intEstOfExprTwo}
There exist constants $C_{22}, \, r_{11} > 0$ such that,
	for all $m\in \mathbb{N}$, all $x\in X^\prime$ and all $r$ satisfying $\sqrt{\frac{(n+1)\log m}{m}} < r < r_{11}$,
		one has
		\begin{equation}
		\Modulus{\int_{y\in B(x,r)} \lambda(x,y) \, \diff V(\sigma(x))(y) - \frac{\pi^n}{m^n}}
			\le \frac{C_{22}}{m^{n+1}}.
		\end{equation}
\end{lemm}

\begin{proof}
For $x\in X^\prime$, $m\in \mathbb{N}$ and $r > 0$, we let
	\begin{equation} \label{exprTwoEqEleven}
\begin{aligned}[t]
	\eta(x, r) &:= \int_{y\in B(x, r)} \lambda(x,y) \, \diff V(\sigma(x))(y) - \frac{\pi^n}{m^n} \\
                    &= \eta_1(x,r) + \eta_2(x,r),
	\end{aligned}
\end{equation}
where
\begin{align}
\eta_1(x,r) := \int_{y\in B(x,r)} \left[\lambda(x,y) - (1-\rho(x,y)^2)^m\right] \, \diff V(\sigma(x))(y) , \\
\eta_2(x,r) := \int_{y\in B(x,r)} \left(1 - \rho(x,y)^2\right)^m \, \diff V(\sigma(x))(y) - \frac{\pi^n}{m^n}.
\end{align}
In terms of the coordinate functions $z : B(x,r_o) \to \C^n$ associated to $\sigma(x) \in \mathcal{G}_x$,
	and identifying $B(x,r)$ with $B(r)$
			(following the notation in \eqref{EuclideanBall} and \eqref{Euclideanvolumeform}),
		it follows readily from Lemma \ref{PtwiseEstOfExprTwo} that,
			if $0 < r < r_{10}$, then
			\begin{equation} \label{exprTwoEqThirteen}
			\Modulus{\eta_1(x,r)} \le C_{19} \int_{B(r)} \left(|z|^2 + m|z|^4\right)
										\left(1 - \frac{1}{2}|z|^2\right)^{m-1} \, \diff V(z).
			\end{equation}
From Lemma \ref{explicitSpherIntegral} (with $k=1,2$, $a = \frac{1}{2}$, and $m$ replaced by $m-1$)
	and noting that $\frac{(m-1)!}{(m-1+n+k)!} \le \frac{1}{m^{n+k}}$, etc.,
	one has, for $r < \sqrt{2}$, 
	\begin{equation} \label{exprTwoEqFourteen}
	\begin{aligned}[t]
	\Modulus{\eta_1(x,r)} &\le
		\frac{C_{19}\pi^n n! (m-1)! 2^{n+1}}{(n-1)! (m+n)!}
			+ m \frac{C_{19}\pi^n (n+1)! (m-1)! 2^{n+2}}{(n-1)! (m+n+1)!} \\
			&\le \frac{C_{23}}{m^{n+1}},
	\end{aligned}
	\end{equation}
	where $C_{23} := C_{19} \pi^n \left(n \cdot 2^{n+1} + n(n+1) 2^{n+2}\right)$.
Similarly, for $r < 1$, one has
\begin{equation} \label{exprTwoEqFifteen}
\begin{aligned}[b]
\eta_2(x,r) &= \int_{B(r)} (1- |z|^2)^m \, \diff V(z) - \frac{\pi^n}{m^n} \\
&= \left(\int_{B(1)} (1- |z|^2)^m \, \diff V(z) - \frac{\pi^n}{m^n}\right)
		-   \int_{B(1) \setminus B(r)} (1- |z|^2)^m \, \diff V(z).
	\end{aligned}	
	\end{equation}
By Lemma \ref{explicitSpherIntegral} again (with $k=0$ and $a=1$), one has
\begin{equation} \label{exprTwoEqSixteen}
0 < \frac{\pi^n}{m^n } - \int_{B(1)} (1 - |z|^2)^m \, \diff V(z) = \frac{\pi^n}{m^n} - \frac{\pi^n m!}{(m+n)!}
	< \frac{\pi^n n(n+1)}{2 m^{n+1}},
\end{equation}
where the last inequality can be obtained by substituting $x_k=\frac{k}{m+k}$ into the
following generalization of Bernoulli's inequality (which follows from a straight-forward induction):
    $\displaystyle\prod_{k=1}^n (1-x_k)\geq 1-\sum_{k=1}^n x_k$ if $0\leq x_1,\ldots, x_n\leq 1$.
From pointwise consideration, one has
\begin{equation} \label{exprTwoEqSeventeen}
\begin{aligned}[t]
0 \le \int_{B(1)\setminus B(r)} (1 - |z|^2)^m \, \diff V(z)
    &\le (1 - r^2)^m \int_{B(1)\setminus B(r)} \, \diff V(z) \\
    &\le (1-r^2)^m \frac{\pi^n}{n!}.
\end{aligned}
\end{equation}
Note that, if $0 < r < 1$, then, upon taking the natural logarithm,
\begin{equation} \label{exprTwoEqEighteen}
(1 - r^2)^m < \frac{1}{m^{n+1}} \quad \Longleftrightarrow \quad -\log(1-r^2) > \frac{(n+1) \log m}{m}.
\end{equation}
Using the fact that $-\log(1-t) > t$ for all $0 < t < 1$,
	one sees that both sides of \eqref{exprTwoEqEighteen} hold if $r > \sqrt{\frac{(n+1) \log m}{m}}$.
For such $r$, it follows from \eqref{exprTwoEqFifteen},
							   \eqref{exprTwoEqSixteen},
							   \eqref{exprTwoEqSeventeen}
							   that
							   \begin{equation} \label{exprTwoEqNineteen}
				\Modulus{\eta_2(x,r)} \le \frac{\pi^n n(n+1)}{2 m^{n+1}} + \frac{\pi^n}{n!} \frac{1}{m^{n+1}}.			
\end{equation}	
Now we let $r_{11} = \min \{r_{10}, 1\} (> 0)$. Combining
            \eqref{exprTwoEqEleven},
			\eqref{exprTwoEqFourteen},
			\eqref{exprTwoEqNineteen},
			it follows that, if $\sqrt{\frac{(n+1)\log m}{m}} < r < r_{11}$,
			then
			\begin{align}
			&\Modulus{\int_{y\in B(x,\rad)} \lambda(x,y) \, \diff V(\sigma(x))(y) - \frac{\pi^n}{m^n} } \\
			\le {} &\frac{C_{23}}{m^{n+1}} +  \frac{\pi^n n(n+1)}{2 m^{n+1}} + \frac{\pi^n}{n!} \frac{1}{m^{n+1}} \notag\\
			= {} & \frac{C_{22}}{m^{n+1}},  \quad{\text{ where }}
					C_{22} := C_{23} +  \frac{\pi^n n(n+1)}{2} + \frac{\pi^n}{n!},\notag \qedhere
			 \end{align}
\end{proof}

\begin{prop} \label{EstimateII}
Let $C_{22}$ and $r_{11}$ be as in Lemma \ref{intEstOfExprTwo}.
Then,
	for all $m\in \mathbb{N}$, all $s\in H^0(X, L^m\otimes E)$ and
                    all $r$ satisfying $\sqrt{\frac{(n+1)\log m}{m}} < r < r_{11}$,
                    one has
        \begin{equation}
        \Modulus{\mathrm{II}} \le  \frac{C_{22}}{m^{n+1}} \norm{s}^2.
        \end{equation}
Here $\mathrm{II}$ is as in \eqref{IIIIII} (with $r$ there as above).
\end{prop}

\begin{proof}
From \eqref{exprTwoEqOne} and \eqref{exprTwoEqSix},
	one easily sees that
	\begin{equation}
		\mathrm{II} =
			\int_{x\in X^\prime}\ptNorm{s}{x}^2   \left[\int_{y\in B(x,\rad)}		
						 \lambda(x,y) \, \diff V(\sigma(x))(y) - \frac{\pi^n}{m^n}
					\right] \Omega(x).
	\end{equation}
Then, by Lemma \ref{intEstOfExprTwo}, one immediately has
	\begin{equation}
	\Modulus{\mathrm{II}} \le \frac{C_{22}}{m^{n+1}} \int_{x\in X^\prime} \ptNorm{s}{x}^2 \Omega(x)
				= \frac{C_{22}}{m^{n+1}} \norm{s}^2. \qedhere
		\end{equation}	
		\end{proof}

\bigskip

\section{Estimation of III and  Proof of Theorem \ref{main-theorem}}\label{Section 6}

\begin{lemm} \label{SchFnOffDiag}
There exists a constant $r_{12}$ with $0 < r_{12} < r_4$ such that,
		for each real number $r$ satisfying $0 < r < r_{12}$ and each $(x,y) \in (X\times X) \setminus W(r)$,
           one has $\Psi_R(x,y) \le 1 - \frac{r^2}{2}$.
\end{lemm}

\begin{proof}
It follows readily from \eqref{4.2} that there exists a constant $r_{13} > 0$ with $r_{13} < r_4$ such that,
    for all $(x,y) \in W(r_{13})$, one has
\begin{equation} \label{schFnOffDiagEqOne}
\Psi_R(x,y) < 1 - \frac{1}{2}\rho(x,y)^2.
\end{equation}
Then by (SGCS-1) for $R$, one has
\begin{equation} \label{schFnOffDiagEqTwo}
\alpha := \sup_{(x,y) \in (X\times X) \setminus W(r_{13})} \Psi_R(x,y) < 1.
\end{equation}
Now one has $r_{12} := \min\{ \sqrt{2(1-\alpha)}, r_{13}\} > 0$, and let $r$ be a number satisfying $0 < r < r_{12}$.
Note that $(X\times X) \setminus W(r) = [(X\times X) \setminus W(R_{13})] \cup [W(r_{13})\setminus W(r)]$.
For any $(x,y) \in X\times X \setminus W (r_{13})$,
    it follows from \eqref{schFnOffDiagEqTwo} and the definition of $r_{12}$ that
    \begin{equation}
    \Psi_R(x,y) \le \alpha \le 1 - \frac{r_{12}^2}{2} < 1- \frac{r^2}{2}.
    \end{equation}
On the other hand, if $(x,y) \in W_{r_{13}} \setminus W(r)$ (so that $\rho(x,y) \ge r$), one also has, from \eqref{schFnOffDiagEqOne},
\begin{equation}
\Psi_R(x,y) \le 1 - \frac{\rho(x,y)^2}{2} \le 1 - \frac{r^2}{2} .
\end{equation}
Hence, for all $(x,y) \in (X\times X) \setminus W(r)$, one has $\Psi_R(x,y) \le 1 - \frac{r^2}{2}$.
\end{proof}

\begin{prop} \label{EstimateIII}
Let $r_{12}$ be as in Lemma \ref{SchFnOffDiag}.
Then there exists a constant $C_{23} > 0$ such that,
    for all $m\in \mathbb{N}$, all $s\in H^0(X, L^m \otimes E)$
            and all $r$ satisfying $\sqrt{\frac{2(n+1)\log m}{m}} < r < r_{12}$, one has
\begin{equation}
\Modulus{\mathrm{III}} \le   \frac{C_{23}}{m^{n+1}} \norm{s}^2.
\end{equation}
Here $\mathrm{III}$ is as in \eqref{IIIIII} (with $r$ there as above).
\end{prop}

\begin{proof}
From the definitions of $\Psi_R, \Psi_P$ in \eqref{diastasis}, \eqref{PPsiP} and similar to \eqref{tempEstOfONE},
     one easily sees that the integral $\mathrm{III}$ in \eqref{IIIIII} satisfies
\begin{equation} \label{exprThreeEqThree}
\Modulus{\mathrm{III}} \le \iint_{(X\times X) \setminus W(r)} \Psi_R(x,y)^m \Psi_P(x,y)
                                            \ptNorm{s}{x} \ptNorm{s}{y}
                                                \Omega(y) \Omega(x).
\end{equation}
By Lemma \ref{SchFnOffDiag}, if $0 < r < r_{12}$, then one has,
    for all $(x,y) \in (X\times X)\setminus W(r)$ and $m\in \mathbb{N}$,
    \begin{equation} \label{exprThreeEqFour}
    \Psi_R(x,y)^m \le \left(1 - \frac{r^2}{2}\right)^m \le e^{-\frac{mr^2}{2}},
    \end{equation}
    where the last inequality follows from the fact that $0 \le 1 - t \le e^{-t}$ for all $0\le t\le 1$.
    By the compactness of the manifold $X$, there exists a constant $C_{24} > 0$ such that
        $0\le \Psi_P(x,y) \le C_{24}$ for all $(x,y)\in X\times X$.
        Together with \eqref{exprThreeEqThree} and \eqref{exprThreeEqFour}, for $0 < r < r_{12}$, one has
        \begin{align}
            \Modulus{\mathrm{III}} &\le C_{24} e^{-\frac{mr^2}{2}} \iint_{(X\times X) \setminus W(r)} \ptNorm{s}{x} \ptNorm{s}{y}
                                                \Omega(y) \Omega(x) \\
                                                &\le C_{24} e^{-\frac{mr^2}{2}} \vol_{\Omega}(X) \norm{s}^2, \notag
        \end{align}
        where the last inequality follows from the Cauchy-Schwarz inequality and
        	$\vol_{\Omega}(X) := \int_X \Omega$.
        By taking the natural logarithm, one has
        \begin{equation}
        e^{-\frac{mr^2}{2}} \le \frac{1}{m^{n+1}} \quad \Longleftrightarrow \quad r \ge \sqrt{\frac{2(n+1)\log m}{m}},
        \end{equation}
        It follows that if $\sqrt{\frac{2(n+1)\log m}{m}} < r < r_{12}$,
        then one has
        \begin{equation}
        \Modulus{\mathrm{III}} \le \frac{C_{23}}{m^{n+1}} \norm{s}^2, \quad {\text{where }} C_{23} := C_{24} \vol_{\Omega}(X). \qedhere
        \end{equation}
\end{proof}

Now we complete the proof of Theorem \ref{main-theorem} as follows:

\medskip
\begin{proof}[Proof of Theorem \ref{main-theorem}]
Let $C_{16}, \, r_9,\, C_{22}, \, r_{11}, \, C_{23}, \, r_{12}$ be as in Proposition \ref{EstimateI}, Proposition \ref{EstimateII} and Proposition \ref{EstimateIII}.
Then it is easy to see that there exists $m_0\in\mathbb{N}$
    such that, for all $m\ge m_0$, one has
\begin{equation}
\sqrt{\frac{2(n+1)\log m}{m}} < \frac{r_9}{m^{\frac{n+2}{2n+5}}} < \min\{r_{11}, r_{12}\}.
\end{equation}
Now, for each $m\ge m_0$, we choose a number $r(m)$ satisfying
\begin{equation} \label{star}
\sqrt{\frac{2(n+1)\log m}{m}} \le r(m) \le \frac{r_9}{m^{\frac{n+2}{2n+5}}}
\end{equation}
(in particular, $r(m)$ may be taken to be one of the two bounds).
Then by \eqref{IntegrationScheme} and \eqref{IIIIII}, Proposition \ref{EstimateI}, Proposition \ref{EstimateII} and Proposition \ref{EstimateIII},
    one sees that for all $m\ge m_0$ and all $s\in H^0(X, L^m\otimes E)$, one has,
        with the number $r$ in $\mathrm{I}, \mathrm{II},\mathrm{III}$ in \eqref{IIIIII} given by $r(m)$ in \eqref{star},
\begin{align}
\Big|  {\bf K}_{R^mP,\Omega}(s,s)    -\dfrac{\pi^n}{m^n}\norm{s}^2\Big|
&   \leq |\mathrm{I}|+|\mathrm{II}|+|\mathrm{III}| \\
\nonumber
&
\leq \dfrac{C_{16} + C_{22} + C_{23}}{m^{n+1}}\norm{s}^2,
\end{align}
By the compactness of the manifold $X$, there exists a constant $C_{25} > 0$ such that,
    for each integer satisfying $1 \le m < m_0$,
\begin{equation}
\Modulus{  {\bf K}_{R^mP,\Omega}(s,s)    -\dfrac{\pi^n}{m^n}\norm{s}^2}
\leq \dfrac{C_{25}}{m^{n+1}}\norm{s}^2,
\end{equation}
    Thus by letting $C = \max \{ C_{16} + C_{22} + C_{23}, C_{25}\} > 0$, one sees that $\eqref{maininequality}$
        holds for all $m\in\mathbb{N}$, and the proof of Theorem \ref{main-theorem} is completed.
\end{proof}

\medskip
Finally we remark that the deduction of Corollary \ref{Corollary 1} from Theorem \ref{main-theorem} can be found in  \cite{CD99}, and thus it will be skipped here.

\end{document}